\documentclass[reqno, a4paper, 12pt]{amsart}
\usepackage{amssymb}
\usepackage[T1]{fontenc}
\usepackage{color, graphicx}
\newtheorem{theorem}{Theorem}[section]
\newtheorem{corollary}[theorem]{Corollary}
\newtheorem{lemma}[theorem]{Lemma}

\theoremstyle{definition}
\newtheorem{definition}[theorem]{Definition}

\theoremstyle{remark}

\numberwithin{equation}{section}
\usepackage{mathptmx}
\usepackage{hyperref}
\hypersetup{unicode=true, pdfstartview=FitH}

\newcommand{\NN}{\mathbb{N}} 
\newcommand{\RR}{\mathbb{R}} 

\newcommand{\pol}{\mathcal{P}}
\newcommand{\hpol}{\mathcal{H}} 
\newcommand{\vdm}{\textsc{vdm}} 
\newcommand{\tell}{\Tilde{\ell}}
\newcommand{\lag}{\mathbf{L}}
\newcommand{\tay}{\mathbf{T}}
\newcommand{\HH}{\mathbb{H}}
\newcommand{\vol}{\mathrm{vol}}
\newcommand{\nn}{\mathbf{n}}

\begin{document}
\title[Continuity property of Lagrange interpolation]{On the continuity of multivariate Lagrange interpolation at Chung-Yao lattices}
\author{Jean-Paul Calvi}
\address{Institut de Mathématiques,
Université de Toulouse III and CNRS (UMR 5219), 31062, Toulouse Cedex 9, France}
\email{jean-paul.calvi@math.ups-tlse.fr}
\author{Phung Van Manh}
\address{Institut de Mathématiques,
Université de Toulouse III and CNRS (UMR 5219), 31062, Toulouse Cedex 9, France and Department of Mathematics, Hanoi University of Education,
136 Xuan Thuy street, Caugiay, Hanoi, Vietnam}
\email{manhlth@gmail.com}



\date{\today}


\begin{abstract} We give a natural geometric condition that ensures that sequences of Chung-Yao interpolation polynomials (of fixed degree) of sufficiently differentiable functions converge to a Taylor polynomial.  
\end{abstract}

\maketitle

\section{Introduction}
\subsection{Stating the problem} When $d+1$ points $a_0,\dots,a_d$ in $\RR$ converge to a limit point $a$, the corresponding Lagrange interpolation polynomial $\mathbf{L}[a_0,\dots,a_d;f]$ of a function $f$ at the $a_i$'s tends to the Taylor polynomial of $f$ at $a$ to the order $d$ and this under the sole assumption that $f$ is $d$ times continuously differentiable on a neighborhood of the limit point. This classical result is an easy consequence of Newton's formula for Lagrange interpolation and of the mean value theorem for divided differences. In this paper, we study a multivariate analogue of this problem. We suppose that the points of a multivariate interpolation lattice $A$ of degree $d$ in $\RR^N$ converge to a limit point $a\in \RR^N$ and ask under what conditions we can assert that the corresponding multivariate Lagrange interpolation polynomials of a function $f$ converge to the Taylor polynomial of $f$ at $a$ to the order $d$ ? The question is answered for a particular but important class of interpolation lattices, the so-called Chung-Yao lattices, see below. 
\subsection{A known criterion} In the multivariate case, a simple clear-cut answer cannot be expected. This perhaps may be regarded as another consequence of the absence of a multivariate mean value equality. We recall a rather general criterion (which actually works for hermitian interpolations) which can be found in \cite{bloomcalvi}. Let us mention that the first results which appeared in the literature concerned the case (of practical importance in finite elements theory) for which the lattices are of the form $A^{(t)}=U^{(t)}(A)$ where $U^{(t)}$ is a sequence of linear transformations whose norms tend to $0$ and $A$ is a fixed lattice. We refer to \cite{bloomcalvi} for details and references to earlier works. 

\medskip

We denote by $\pol^d(\RR^N)$ the space of polynomials in $N$ real variables of degree at most $d$, $X^\alpha$ is the monomial function corresponding to the $N$-index $\alpha$, that is $X^\alpha(x)=x_1^{\alpha_1}\cdots x_N^{\alpha_N}$ for $x=(x_1,\dots,x_N)\in\RR^N$. The length of $\alpha$ is the degree of $X^\alpha$, $|\alpha|=\sum_{i=1}^N \alpha_i$. We denote by $m_d$ the dimension of the vector space $\pol^d(\RR^N)$. We have $m_d=\binom{N+d}{d}$. In the whole paper, $N\geq 2$. 

\begin{theorem}[Bloom and Calvi]\label{theo_old}
 Let $A^{(s)}$ be a sequence of interpolation lattices of degree $d$ in $\RR^N$. If the following condition holds
\begin{equation}\label{eq:hypocalvibloom} |\alpha|=d+1\Longrightarrow \lim_{s\to \infty}\mathbf{L}[A^{(s)}\,;\, X^{\alpha}]=0, \end{equation}
then, for every function $f$ of class $C^{m_d-1}$ in a neighborhood of the origin $0$, we have 
\begin{equation}\lim_{s\to \infty} \lag[A^{(s)}\,;\, f]=\tay^d_0 (f), \end{equation}
where $\lag[A^{(s)}\,;\, \cdot ]$ (resp. $\tay^d_0 (\cdot)$) denotes the Lagrange interpolation projector at the points of $A^{(s)}$ (resp. the Taylor projector at $0$ of 
order $d$). 
\end{theorem}  

Unfortunately condition (\ref{eq:hypocalvibloom}) is not easy to verify, especially if the
degree of interpolation is not small, and it seems difficult to check it on general classes of interpolation lattices.  Besides, theorem \ref{theo_old} requires a high order of smoothness. We point out, however, that although whether the level of differentiability  required in theorem \ref{theo_old} is optimal is not known (in the case of Lagrange interpolation), examples do exist for which convergence does not hold for function of class $C^{d+1}$ but holds for function of higher smoothness, see \cite[example 5.4]{bloomcalvi}.    

\medskip

The aim of this paper is to give a natural geometric condition in the case where the interpolation lattices are Chung-Yao lattices. From an algebraic point of view, they can be regarded as the simplest interpolation lattices : every point is situated at the intersection of $N$ hyperplanes chosen among a
minimal family and the corresponding Lagrange fundamental polynomials are products of 
affine forms. The definition and main properties of Chung-Yao lattices are collected in section \ref{sec:chungyaolat}. Our criterion is given and commented in section \ref{sec:mainth}. The proof is quite technical and is postponed to the next section. It relies on a remainder formula due to Carl de Boor. 

\medskip

We need very few facts from general interpolation theory. They are recalled in the following subsection.  

\subsection{Basic facts on interpolation} 
Let $E$ be a $m$-dimensional space of functions on $\RR^N$ and $A=\{a_1,\ldots, a_{m}\}\subset \RR^N$. We say that $A$ is an interpolation lattice for $E$ if for every function $f$ defined on $A$ there exists a unique $L\in E$ such that $L=f$ on $A$. Given a basis $\mathbf{f}=(f_1,\ldots,f_m)$ of $E$, we define the Vandermonde determinant $\vdm(\mathbf{f}\,;\, A)$ by
\begin{equation}\vdm(\mathbf{f}\,;\, A):=\det \big(f_i(a_j)\big)_{i,j=1}^m. \end{equation}
Then $A$ is an interpolation lattice if and only if \begin{equation}\label{eq:vandercond}  \vdm (\mathbf{f}\,;\, A)\neq 0.\end{equation} Of course, the condition is independent from the choice of the basis $\mathbf{f}$. 
When (\ref{eq:vandercond}) is satisfied, we have
\begin{equation} L=\sum_{i=1}^{m} f(a_i) \; \mathbf{l}(A,a_i,\cdot), \end{equation}
where $\mathbf{l}(A,a_i, \cdot)$ is the unique element of $E$ which vanishes on $A\setminus\{a_i\}$ and takes the value $1$ at $a_i$, 
\begin{equation} \mathbf{l}(A,a_i,x)=\frac{\vdm(\mathbf{f}\,;\, \{a_1,\dots, a_{i-1},x,a_{i+1},\ldots,a_{m}\})}{\vdm(\mathbf{f}\,;\, A)}, \quad 1\leq i\leq m,\quad x\in \RR^N. \end{equation}
In the case where $E=\pol^d(\RR^N)$ we write $L=\lag[A;f]$ and call it the Lagrange interpolation of $f$ at $A$. We say that $A$ is an interpolation lattice of degree $d$. The only other case that we consider in this paper is $E=\hpol^d(\RR^N)$, the space of homogeneous polynomials of degree $d$ in $N$ variables whose dimension is $\binom{N+d-1}{d}$.   
  
 \section{Chung-Yao lattices}\label{sec:chungyaolat}
 
  We recall the construction of the lattices and of some objects attached to them. Despite their apparent simplicity, it seems that these configurations were first considered in 1977's Chung and Yao's paper \cite{chungyao}. Here, we essentially follow the presentation and notational conventions
 of de Boor \cite{deboor}. 
 
\medskip
  
 We work  in $\RR^N$ endowed with its canonical euclidean structure. The corresponding scalar product is denoted  by $\langle \cdot , \cdot \rangle$.  
 
\medskip

 A set of $N$ hyperplanes $H=\{\ell_1,\dots, \ell_N\}$ in $\RR^N$ is said to be in \emph{general position} if the intersection of the $N$ hyperplanes is a singleton, that is $$\bigcap_{i=1}^N\ell_i=\{\vartheta_H\}.$$ If $\ell_i=\{x\in \RR^N \;:\; \langle n_i,x\rangle=c_i\}$, $i=1,\dots,N$, then $H$ is in general position if and only if $\det(n_1,\dots,n_N)\ne 0$. 
 
\begin{definition} A collection $\HH$ of (at least $N$) distinct hyperplanes in $\RR^N$ is said to be in general position if 
\begin{enumerate}
	\item Every $H\in\binom{\HH}{N}$ --- i.e. every subset of $N$ hyperplanes in $\HH$ --- is in general position (as defined above).
	\item The map  
	\begin{equation} H\in\binom{\HH}{N} \longmapsto \vartheta_H:=\bigcap_{\ell\in H} \ell \in \RR^N \end{equation}
	is one-to-one. Here and in the sequel we confuse the singleton $\bigcap_{\ell\in H} \ell$ with its element. 
\end{enumerate}
\end{definition} 

This definition stands at the basis of the following result.  
    
\begin{theorem}[Chung and Yao \cite{chungyao}] Let $\HH$ be a set of $d\geq N$ hyperplanes in general position in $\RR^N$.  
The lattice
 \begin{equation} \Theta_{\HH}=\left\{\vartheta_H=\bigcap_{\ell\in H} \ell \;:\;  H\in\binom{\HH}{N} \right\}\end{equation}
 is an interpolation lattice of degree $d-N$. Moreover, if $\ell\in \HH$ is given by $\ell=\{x\in \RR^N \;:\; \langle n_\ell,x\rangle=c_\ell\}$ then we have the interpolation formula
 \begin{equation}\label{eq:chinterpform} \lag[\Theta_\HH\,;\, f](x)=\sum_{H\in\binom{\HH}{N}} f(\vartheta_H)\; \prod_{\ell\not\in H} \frac{\langle n_\ell,x \rangle - c_\ell}{\langle n_\ell,\vartheta_H\rangle -c_\ell}. \end{equation}
The lattice $\Theta_\HH$ is called a \emph{Chung-Yao lattice} (of degree $d-N$) and the interpolation formula is called the \emph{Chung-Yao interpolation formula} corresponding to $\HH$. In particular, we have
\begin{equation} \mathbf{l}(\Theta_\HH, \vartheta_H, x)= \prod_{\ell\not\in H} \frac{\langle n_\ell,x \rangle - c_\ell}{\langle n_\ell,\vartheta_H\rangle -c_\ell}, \quad H\in\binom{\HH}{N}. \end{equation} 
\end{theorem}    

When the set of hyperplanes we use is clear, we write $\Theta$ instead of $\Theta_\HH$. Of course, in (\ref{eq:chinterpform}), different equations for the hyperplanes yield a same formula. In the particular case $N=1$, every set of interpolation nodes may be regarded as a (trivial) Chung-Yao lattice.

\medskip

As shown by (\ref{eq:chinterpform}), interpolation polynomials at Chung-Yao lattices are easy to compute. Some difficulties, however, must be pointed out. 
In constructing a Chung-Yao lattice, we start from a family of hyperplanes and compute the interpolation points by solving, in principle, $m_d$ linear 
systems (of order $N$). Besides, it is a difficult problem, even in the case $N=2$, to decide how to choose the hyperplanes if a special requirement is
made on the location of the interpolation points. For instance, we currently do not know what kind of limiting distribution we can obtain with a growing
number of Chung-Yao points. We mention that an interesting Chung-Yao lattice was constructed by Sauer and Xu \cite{sauerxu} on bidimensional disks.

\section{Chung-Yao lattices of points converging to the origin}\label{sec:mainth}
\subsection{The convergence theorem} From now on, we shall confuse an hyperplane $\ell$ with the affine form $\ell(x)=\langle \nn, x\rangle -c$ which defines it, where $\nn$ is normalized so that $\|\nn\|=1$. This abuse of language (each hyperplane has two normalized equations) should not create confusion. Boldfaced $\nn$ will be kept for normalized vectors and vectors derived from them.  
\par 
Supposing that the points of a sequence $\Theta^{(s)}$ of Chung-Yao lattices of same degree converge to the origin (or to any other fixed point), we study under what conditions the corresponding interpolation operator converges to the Taylor projector at the origin. Our main result is summarized in the following theorem.

\begin{theorem}\label{th:main} Let $d\geq N$. Let $\Theta^{(s)}$, $s\in \NN$, be a sequence of Chung-Yao lattices of degree $d-N$ in $\RR^N$. We assume that $\Theta^{(s)}$ is the lattice given by the family of hyperplanes  \begin{equation}\label{eq:defparaell} \HH^{(s)}=\{\ell^{(s)}_1, \dots, \ell^{(s)}_d\}, \quad \textrm{with $\ell^{(s)}_i=\langle \nn^{(s)}_i, \cdot \rangle -c^{(s)}_i$, $\|\nn_i^{(s)}\|=1$, $i=1,\dots,d$}.\end{equation}
Consider the following two conditions. 
\begin{enumerate}
	\item[(C1)]\label{it:tendtozero} All the points of the lattice tend to $0$ as $s\rightarrow\infty$, that is $\max \{\|\vartheta\|\,:\, \vartheta\in \Theta^{(s)}\}\rightarrow 0$ as $s\rightarrow\infty$,
	\item[(C2)]\label{it:volume} The volumes \begin{equation}\vol\left(\nn_{i_1}^{(s)}, \dots, \nn_{i_N}^{(s)}\right), \quad 1\leq i_1<i_2<\cdots < i_N\leq d,\end{equation} of the parallelotope spanned by the vectors $\nn_{i_1}^{(s)}, \dots, \nn_{i_N}^{(s)}$ are bounded from below, away from $0$, uniformly in $s$. 
\end{enumerate}
If conditions (C1) and (C2) are satisfied then, for every function $f$ of class $C^{d-N+1}$ on a neighborhood of the origin, we have 
\begin{equation}\label{eq:conclumain}\lim_{s\rightarrow\infty} \lag [\Theta^{(s)}\,;\, f]= \tay_0^{d-N}(f).\end{equation}
\end{theorem}

Of course, (\ref{eq:conclumain}) holds in every normed vector space topology of $\pol^{d-N}(\RR^N)$.

\subsection{On condition (C2)} The condition on the volume of the parallelotopes is equivalent to the following,

\begin{equation}\label{eq:C2withdet} \liminf_{s\rightarrow\infty} \min_{1\leq i_1<\dots<i_N\leq d} \left|\det \left(\nn_{i_1}^{(s)}, \dots, \nn_{i_N}^{(s)}\right)\right| >0. \end{equation}

In $\RR^2$ we have  \begin{equation}\label{eq:anglecond} \vol(\nn_i^{(s)},\nn_j^{(s)})=\sin (\alpha_{ij}^{(s)}),\end{equation} where $\alpha_{ij}^{(s)}\in ]0,\pi[$ is the line angle between the lines $\ell_i$ and $\ell_j$. Thus $\HH$ satisfies condition (C2) if and only of the angles between any two (distinct) lines in $\HH^{(s)}$ remain uniformly bounded from below by a positive constant. An example of Chung-Yao lattice of degree $2$ in $\RR^2$ and the various parameters involved in theorem \ref{th:main} are shown in figure \ref{fig:CYfig}.
\begin{figure}[htb]
	\centering
		\includegraphics[width=0.8\textwidth]{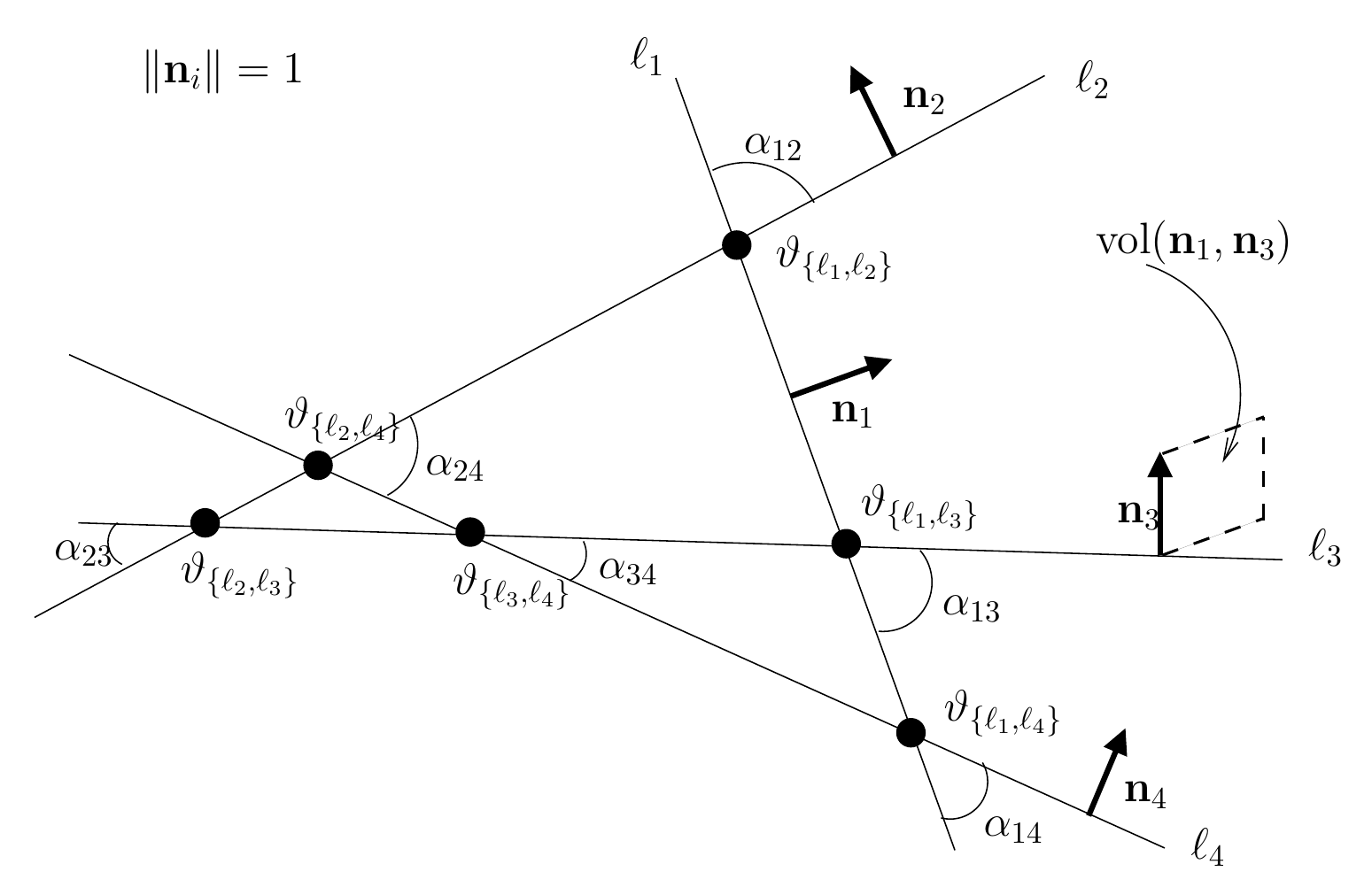}
\caption{A bidimensional Chung-Yao lattice.}
	\label{fig:CYfig}
\end{figure}

\medskip

Conditions (C1) and (C2) are obviously independent. However, when we know that the second one holds true, the first one
is easily checked as shown by the following lemma. 

\begin{lemma}\label{algebra_condition}
If (C2) is satisfied then (C1) is equivalent to 
\begin{enumerate}
\item[(C3)]  $\lim_{s\rightarrow\infty} \max_{i=1,\dots,d} |c^{(s)}_i|= 0$ where $c^{(s)}_i$ is defined in (\ref{eq:defparaell}).
\end{enumerate}
\end{lemma}

\begin{proof} We show that (C1) implies (C2). Consider $H^{(s)}\in \binom{\HH^{(s)}}{N}$ with $\ell_i^{(s)}\in H^{(s)}$. From $\langle \nn_i^{(s)}, \vartheta_{H^{(s)}}\rangle -c_i^{(s)}=0$, we get $$|c_i^{(s)}|\leq \| \nn_i^{(s)}\|\cdot  \| \vartheta_{H^{(s)}}\| =  \| \vartheta_{H^{(s)}}\|  \rightarrow 0, \quad s\to\infty.$$ 
To show the converse, we observe that  if  $H^{(s)}=\{\ell^{(s)}_{i_1},\dots,\ell_{i_N}^{(s)}\}$ then the coordinates $(x_k)$ of $\vartheta_{H^{(s)}}$ are solutions of the linear system
$$\sum_{k=1}^N \nn^{(s)}_{i_j \,k}x_k=c^{(s)}_j, \quad j=1,\dots, N,$$  
and the claim follows from Cramer's formula in which, thanks to condition (C2), the denominator remains away from $0$ whereas the numerator tends to $0$.\end{proof}
\subsection{Affine transformations of Chung-Yao lattices} 
 Let $\mathcal{L}(x)=L(x)+b$ be an affine transformation (isomorphism) of $\RR^N$ (with $L$ its linear part). If $\HH$ is in general position so is $\mathcal{L}(\HH):=\{\mathcal{L}(\ell_i)\,:\, i=1,\dots,d\}$ and $\mathcal{L}$ induces a one-to-one correspondence between $\binom{\HH}{N}$ and $\binom{\mathcal{L}(\HH)}{N}$. Moreover
if $H\in \binom{\HH}{N}$ then $$\vartheta_{\mathcal{L}(H)}=\mathcal{L}(\vartheta_{H})\quad\textrm{and}\quad \Theta_{\mathcal{L}(\HH)}=\mathcal{L}(\Theta_\HH).$$

In the following theorem we translate the conditions of theorem \ref{th:main} when the points of a Chung-Yao lattice are sent to the origin by applying a sequence of affine transformations. 

\begin{theorem}\label{th:afftrans} Let $\HH=\{\ell_1,\dots,\ell_d\}$ be a fixed collection of $d$ hyperplanes in general position in $\RR^N$, $d\geq N$, with, as above, $\ell_i=\{x\in\RR^N \,:\, \langle \nn_i,x\rangle -c_i=0\}$, $\|\nn_i\|=1$. Let $\mathcal{L}_s=L_s+b_s$, $s\in \NN$,  be a sequence of affine transformations of $\RR^N$. We set \begin{equation} \HH^{(s)}=  \mathcal{L}_s(\HH), \quad s\in \NN. \end{equation}
We consider the sequence of Chung-Yao lattices $\Theta^{(s)}$ induced by $\HH^{(s)}$. The following assertions are equivalent.
\begin{enumerate}
	\item $\Theta^{(s)}$ satisfies conditions (C1) and (C2). 
	\item There exists a positive constant $\Delta$ such that
	\begin{equation} \left|\det L_s\right| \cdot \prod_{j=1}^N \|L_s^{-T}(\nn_{i_j})\| \leq \Delta, \quad 1\leq i_1 <\dots < i_N\leq d, \quad s\in \NN ,\end{equation} 
	and 
\begin{equation}\max_{i=1,\dots,d} \frac{1}{\|L_s^{-T} (\nn_i)\|}\cdot \big|c_i+ \langle \nn_i, L_s^{-1}(b_s)\rangle\big| \rightarrow 0,\quad s\rightarrow\infty,\end{equation} 
\end{enumerate}
where $L_s^{-T}$ denotes the transpose of the inverse of $L_s$. 
\end{theorem}

\begin{proof} It follows from the normalized equation of $\mathcal{L}_s(\ell_i)$ together with lemma \ref{algebra_condition}. Indeed, with $\ell_i(x)=\langle \nn_i, x\rangle -c_i$, we have
$$\mathcal{L}(\ell_i) =\{x\in\RR^N \,:\,  \langle \nn_i, \mathcal{L}_s^{-1}(x)\rangle -c_i=0\}.$$ Since for $x\in \mathcal{L}(\ell_i)$,
\begin{multline}0=\langle \nn_i, L_s^{-1}(x-b_s))\rangle -c_i
=\langle \nn_i, L_s^{-1}(x)\rangle -\big(c_i+\langle \nn_i, L_s^{-1}(b_s)\rangle\big)\\
=\langle L_s^{-T}(\nn_i), x\rangle -\big(c_i+\langle \nn_i, L_s^{-1}(b_s)\rangle\big),\end{multline}
a normalized equation of $\mathcal{L}_s(\ell_i)$ is given by
$$\left\langle \frac{L_s^{-T}(\nn_i)}{\|L_s^{-T}(\nn_i)\|}\, , x\right\rangle - \frac{1}{\|L_s^{-T}(\nn_i)\|}\big\{c_i+\langle \nn_i, L_s^{-1}(b_s)\rangle\big\}.\qedhere$$
\end{proof} 

\subsection{Examples} In $\RR^2$ any interpolation lattice of degree $1$ is a Chung-Yao lattice (based on the three distinct lines defined by the interpolation points). Moreover, any such lattice is the image under an affine isomorphism of the lattice $\Theta:=\{(0,0), (1,0), (0,1)\}$ constructed with the lines of equations $\ell_1(x_1,x_2)=x_1$, $\ell_2(x_1,x_2)=x_2$ and $\ell_3(x_1,x_2)=x_1+x_2-1$. 

Consider the affine transformations $\mathcal{L}_s$ defined by
$$\mathcal{L}_s(x)=\begin{pmatrix} t^2 & 0 \\ 0 & -t^2u\end{pmatrix}\begin{pmatrix} x_1\\ x_2\end{pmatrix} +\begin{pmatrix} t \\ t\end{pmatrix}, \quad x=\begin{pmatrix} x_2\\ x_2\end{pmatrix}\in \RR^2,\quad t=1/s, \quad s\in \NN^\star,$$
where $u$ is a function of $t$ such that $\lim_{t\rightarrow 0}u(t)=1$,
and the lattice 
$$\Theta^{(s)}=\mathcal{L}_s(\Theta)=\left\{(t,t), (t^2+t,t), (t,-t^2u+t)\right\}, \quad t=1/s.$$
It is not difficult to see that $\Theta^{(s)}$ satisfies condition (C1) and (C2). For (C2) we use (\ref{eq:anglecond}) and observe that one of the angle is equal to $\pi/2$ while, thanks to the assumption on $u$, the other two tend to $\pi/4$ as $t\rightarrow 0$. Hence, according to theorem \ref{th:main}, the corresponding Lagrange interpolation polynomials at $\Theta^{(s)}$ of any twice continuously differentiable function $f$ on a neighborhood of $0$ converge to the Taylor polynomial of $f$. This example shows that the assumptions of theorem \ref{th:main}, even in the simple case of theorem \ref{th:afftrans}, are weaker that those given in \cite[proposition 2.1]{bloomcalvi}. Indeed the assumption $$\|(t,t)\|^2 \cdot \left|\ell\left(\Theta^{(s)}, (t,t), \cdot\right)\right| \rightarrow 0, \quad t\to 0,$$
is required in that proposition whereas it clearly does not hold here since, as is easily checked, 
$$\ell\left(\Theta^{(s)}, (t,t), x\right)={\frac{x_2-u\,x_1+\left(t^2+t\right)\,u-t}{t^2\,u}}.$$

\medskip

We now give an example showing that convergence to the Taylor projector no longer holds, in general, when condition (C2) is not satisfied. We use a computation done in \cite[example 1.2.]{bloomcalvi}. We fix $\epsilon \geq 0$ and define  $$\Theta_{\HH^{(s)}}=\left\{(0,0),(t,t^{2+\varepsilon}), (2t,0)\right\}\subset \RR^2, \quad t=1/s,\quad s\in \NN^\star.$$  
This lattice satisfies (C1) but not (C2) and it is readily checked that 
$$\lag \left[\Theta_{\HH^{(s)}}\, ;\, X^{(2,0)}\right](x)=2tx_1-\frac{x_2}{t^\varepsilon}, $$
which clearly does not converge to $\tay^1_0(X^{(2,0)})=0$ as $s=1/t\rightarrow\infty$. The case $\epsilon=0$ shows that that the Lagrange polynomials may converge to a limit different from the Taylor polynomial.

\section{Further properties of Chung-Yao lattices and proof of theorem \ref{th:main}}
\subsection{de Boor's identity} 
In the following $\HH$ always denotes a set of $d\geq N$ hyperplanes in general position in $\RR^N$ and $\Theta=\Theta_\HH$ the corresponding Chung-Yao lattice. We will always assume that
\begin{equation} \HH=\{\ell_1, \dots, \ell_d\}. \end{equation}
The elements of $\HH$ are ordered according to the indexes. Every subset of $\HH$ is endowed with the induced ordering.
  
  \smallskip
  
  If $K$ is a subset of $N-1$ elements in $\HH$, that is $K\in \binom{\HH}{N-1}$, 
 then $\cap_{\ell\in K} \ell$ is a line in $\RR^N$ which contains $d-N+1$ points of $\Theta$. Indeed, it passes through every $\vartheta_H$ such that 
 $H\in \binom{\HH}{N}$, $K\subset H$. The set of these $d-N+1$ points is denoted by $\Theta_K$, 
 \begin{equation}\label{eq:defthetaK} \Theta_K=\Theta \cap \left(\bigcap_{\ell\in K} \ell\right), \quad K\in \binom{\HH}{N-1}. \end{equation}
 
Assume that $K=\{\ell_{i_1}, \dots, \ell_{i_{N-1}}\}$ with $i_1<i_2<\dots<i_{N-1}$. Since the map
\begin{equation}\label{eq:deflinforfornK} v\in \RR^N \mapsto \det(v, \nn_{i_1}, \dots, \nn_{i_{N-1}})\end{equation}
is a linear form, there exists a vector, which we denote by $\nn_K$, such that
  \begin{equation}\label{eq:defnK} \det(v,\nn_{i_1}, \dots, \nn_{i_{N-1}})=\langle v , \nn_K\rangle, \quad v\in \RR^N. \end{equation}
 
As defined, the value of $\nn_K$ depends on the ordering of the hyperplanes of $K$. A different ordering may change $\nn_K$ in $-\nn_K$. It is to avoid further discussion of this detail that we assumed we start with a particular ordering of $\HH$ and agreed that every subset of $\HH$ is endowed with the induced ordering.
 
\begin{lemma} If $K=\{\ell_{i_1}, \dots, \ell_{i_{N-1}}\}$ then the direction of the line  $\cap_{\ell\in K} \ell$ is given by the (nonzero) vector $\nn_K$. We have $\|\nn_K\|\leq 1$. \end{lemma}
\begin{proof} The first claim is a consequence of the equations  $$\langle \nn_{i_j}, \nn_K\rangle=0, \quad j=1,\ldots, N-1,$$ which follows readily from (\ref{eq:defnK}). Next, by Hadamard's inequality, the norm of the linear form (\ref{eq:deflinforfornK}) is smaller than the product of the $\|\nn_{i_j}\|$'s which is smaller than one. Hence, in view of (\ref{eq:defnK} ), so is the norm of $\nn_K$.  \end{proof}

The vectors $\nn_K$ play a fundamental role in our proof of theorem \ref{th:main}. 

Note in particular that if $H\in \binom{\HH}{N}$ and $\ell\in H$ then we may speak of $\nn_{H\setminus \ell}$. From now on, we use $H\setminus \ell_i$ for $H\setminus \{\ell_i\}$. 
\begin{lemma}[de Boor's identity] If $H\in \binom{\HH}{N}$ then we have 
\begin{equation}\label{coordinate}
x=\vartheta_H+\sum_{\ell\in H}\frac{\ell (x)}{\tell(\nn_{H\setminus \ell})}\; \nn_{H\setminus \ell}, \quad x\in\RR^N,\end{equation}
where $\tell$ denotes the linear part of $\ell$ (thus $\tell(x)=\langle \nn, x\rangle$ if $\ell(x)=\langle \nn, x \rangle-c$). In particular, for every $H$, the vectors $\nn_{H\setminus \ell}$, $\ell\in H$, form a basis of $\RR^N$. 
\end{lemma}
 
 \begin{proof} See \cite[p. 37]{deboor}. \end{proof}

\subsection{de Boor's remainder formula}\label{sec:deBoorrf}
We now recall the definition of multivariate divided differences. Let $\Omega$ be an open convex set in $\RR^N$, to every set $A=\{a_0, \ldots,a_s\}\subset\Omega$ (the points are not necessarily distinct) and $f\in C^s(\Omega)$, we associate a $s$-linear form on $(\RR^N)^s$ defined by
\begin{multline}(\RR^N)^s\ni (v_1,\dots,v_s)  \longmapsto \\ [a_0,\ldots,a_s|v_1,\dots,v_s]f:=\int\limits_{[A]}D_{v_1}\ldots D_{v_s}f=\int\limits_{[A]}f^{(s)}(\cdot)(v_1,\dots,v_s), \end{multline}
where $f^{(s)}$ denotes the $s$-th total derivative of $f$,
$$\int\limits_{[A]}g=\int\limits_{\Delta_s} g\left(a_0+\sum_{i=1}^s \xi_i(a_i-a_0)\right)d\xi_1\dots d\xi_s$$
and $\Delta_s$ is the standard simplex $\{\xi=(\xi_1,\ldots,\xi_s): \xi_i\geq 0,\; \sum_{i=1}^s \xi_i\leq 1\}.$ This symmetric $s$-linear form is called the \emph{multivariate divided difference} of $f$ at $A$. Note that, when $f\in C^s(\Omega)$ is fixed, the function
$$\Omega^{s+1}\times (\RR^N)^s\ni (a_0,\ldots,a_s , v_1,\ldots,v_s)\longmapsto [a_0,\ldots, a_s\,|\, v_1,\ldots,v_s]f$$
is continuous (as a function of its two groups of variables). 
\par
We now state a beautiful error formula due to Carl de Boor.
\begin{theorem}[de Boor's remainder formula]\label{errordeboor}
Let $\HH=\{\ell_1,\dots,\ell_d\}$ be a collection of $d\geq N$ hyperplanes in general position in $\RR^N$ and $\Theta=\Theta_\HH$ the corresponding Chung-Yao lattice.  For $K\in \binom{\HH}{N-1}$, we define the polynomial $P_K$ of degree $d-N+1$ by the relation
\begin{equation}\label{eq:defPKdBoor} P_K(x)=\prod_{\ell\in \HH\setminus K}\frac{\ell(x)}{\tell (\nn_K)}, \end{equation}
where, as above, $\tilde{\ell}$ is used for the linear part of $\ell$. 
 \par
 The error between a function $f$ of class $C^{d-N+1}$ on a convex neighborhood $\Omega$ of $\Theta$ and the Lagrange interpolation polynomial of $f$ at $\Theta$ is given by the following formula. 
\begin{equation}
f(x)=\lag[\Theta; f](x)+\sum_{K\in\binom{\HH}{N-1}}P_K(x)\,\cdot\, \Big[\Theta_{K},x\, |\, \underbrace{\nn_K,\cdots,\nn_K}_{d-N+1}\Big]f,\quad x\in \Omega. \end{equation}
Recall that for $K\in \binom{\HH}{N-1}$, $\Theta_{K}$ is the subset formed by the $d-N+1$ points of $\Theta$ lying on the line $\cap_{\ell\in K}\ell$, see (\ref{eq:defthetaK}).
\end{theorem}
\begin{proof} See \cite[theorem 3.1.]{deboor}. 
\end{proof} 
\subsection{Some algebraic identities}
We now prove two auxiliary lemmas. The first one (lemma \ref{homo_unisolvent}) shows that the $\nn_K$'s, $K\in\binom{\HH}{N-1}$, themselves form a certain interpolation lattice. The second one (lemma \ref{general_rep}) is a somewhat mysterious representation formula for symmetric multilinear forms.  
\begin{lemma}\label{homo_unisolvent}
Let $\HH=\{\ell_1,\dots,\ell_d\}$ be a collection of $d$ hyperplanes in general position in $\RR^N$ with $d\geq N$. The set \begin{equation} \mathcal V:=\left\{\nn_K: K\in\binom{\HH}{N-1}\right\}\end{equation} is an interpolation lattice for the space $\hpol^{d-N+1}(\RR^N)$ of homogeneous polynomials of degree $d-N+1$.
\end{lemma}
\begin{proof} It suffices to prove the following two assertions. 
\begin{enumerate}
	\item The cardinality of $\mathcal{V}$ is equal to the dimension of $\hpol^{d-N+1}(\RR^N)$ which is $\binom{d}{d-N+1}=\binom{d}{N-1}$. 
	\item For every $\nn_K$ in $\mathcal{V}$ there exists $H_K\in \hpol^{d-N+1}(\RR^N)$ such that $H_K(\nn_K)=1$ but $H_K$ vanishes on $\mathcal V\setminus\{\nn_K\}$  
\end{enumerate}

\smallskip

To verify the first point, we just need to check that if $K, K'\in \binom{\HH}{N-1}$ and $K\ne K'$ then $\nn_{K}\ne \nn_{K'}$. But, if $K\ne K'$ there exists $\ell\in K\setminus K'$ with $\ell(x)=\langle \nn, x\rangle - c$. Assume that $K'=\{\ell_{i_1}, \dots, \ell_{i_{N-1}}\}$. Since $\ell \cup{K'}$ is a set of $N$ hyperplanes in general position, we have $\det(\nn, \nn_{i_1}, \dots, \nn_{i_{N-1}}) \neq 0$ hence, in view of (\ref{eq:defnK}), $\langle \nn, \nn_{K'}\rangle\neq 0$. On the other hand, since $\ell\in K$, $\langle \nn, \nn_{K}\rangle=0$. Hence $\nn_{K}\ne \nn_{K'}$. 

\smallskip

As for the second point, for $K\in \binom{\HH}{N-1}$, we set 
\begin{equation}\label{eq:defpK} H_K:= \tilde{P}_{K}(x)=\prod_{\ell\in \HH\setminus {K}}\frac{\tell(x)}{\tell (\nn_{K})}.\end{equation} 
This clearly defines a homogeneous polynomial of degree $d-N+1$ in $\RR^N$ satisfying $H_K(\nn_K)=1$. Moreover, if $K'\in \binom{\HH}{N-1}$, $K'\ne K$, then we can find $\ell$ in $(\HH\setminus K)\cap K'$. Since $\ell\not\in K$, the factor $\tilde{\ell}(\nn_{K'})$ appears in $H_{K}(\nn_{K'})$. However, since $\ell\in K'$, $\tilde{\ell}(\nn_{K'})=\langle \nn, \nn_{K'}\rangle =0$. Hence $H_{K}(\nn_{K'})=0$.
\end{proof}
The interpolation formula corresponding to the interpolation lattice --- and using the polynomials $H_K=\tilde{P}_K$ in (\ref{eq:defpK}) --- yields the following identity. 
 
\begin{corollary}\label{homo_rep}
With the assumptions of the lemma, for every  symmetric $(d-N+1)$-linear form $\phi$ on $\RR^N$, we have
\begin{equation} \phi(v^{d-N+1})=\sum_{K\in\binom{\HH}{N-1}}\tilde{P}_K(v)\,\cdot\,\phi(\nn_{K}^{d-N+1}),\quad v\in\RR^N, \end{equation}
 where we use $u^{d-N+1}:=(u,\dots,u)$ ($d-N+1$ times).
\end{corollary}

\begin{lemma}\label{general_rep}
Let $\HH=\{\ell_1,\ldots,\ell_d\}$ be a collection of hyperplanes in $\RR^N$ in general position with $d\geq N$. We set 
\begin{equation} \HH_i= \{\ell_1,\ldots,\ell_{i}\}, \quad 1\leq i\leq d,\end{equation}
and  
\begin{equation} P^{[i-1]}_K(x)=\prod_{\ell\in \HH_{i-1}\setminus K}\frac{\ell(x)}{\tell(\nn_K)},\quad K\in \binom{\HH_{i-1}}{N-1}, \quad N\leq i \leq d+1. \end{equation} 
Then for every symmetric $(d-N+1)$-linear form $\phi$ on $\RR^N$, we have
\begin{equation}\label{eq:kindofnewton} \phi(x^{d-N+1})=\sum_{i=N}^{d+1}\sum_{K\in \binom{\HH_{i-1}}{N-1}} P^{[i-1]}_K(x)\, \cdot \, \phi\Big(x^{d-i}\,,\; \vartheta_{K\cup\ell_i}\,,\,\nn_K^{i-N}\Big), \quad x\in \RR^N.\end{equation}
\end{lemma}
In the above formula, we agree that when $d-i$ (resp. $i-N$) is not positive then $x$ (resp. $\nn_K$) does not appear in $\phi(x^{d-i},\vartheta_{K\cup\ell_i},\nn_K^{i-N})$,
	 and when $i=d+1$ then $x$ and $\vartheta_{K\cup\ell_i}$ do not appear.
 Likewise, if the product in the definition of $P^{[i-1]}_K$ is empty then its value is taken to be $1$.  
 
\smallskip

We need the following simple observation.

\begin{lemma}\label{th:techobser}Let $\HH=\{\ell_1,\ldots,\ell_d+1\}$ be a collection of hyperplanes in $\RR^N$ in general position with $d\geq N$. As above, we write $\HH_d=\{\ell_1,\ldots,\ell_d\}$. Let $K'\in\binom{\HH_d}{N-2}$. If $K\in\binom{\HH_d}{N-1}$ and $K'\nsubseteq K$ then $\tilde{P}^{[d]}_K(\nn_{K'\cup\ell_{d+1}})=0$ where $$\tilde{P}^{[d]}_K= \prod_{\ell\in \HH_{d}\setminus K}\frac{\tilde{\ell}(\cdot)}{\tell(\nn_K)}.$$
\end{lemma}

\begin{proof} Take $\ell_i\in K'\cap (\HH_d\setminus K)$. The fact that $\ell_i\in K'$ gives $$0=\langle \nn_i, \nn_{K'\cup \ell_{d+1}}\rangle= \tilde{\ell}_i(\nn_{K'\cup \ell_{d+1}})$$ and since $\ell_i\not\in K$,  $\tilde{\ell}_i(\nn_{K'\cup \ell_{d+1}})$ is a factor of 
$\tilde{P}^{[d]}_K(\nn_{K'\cup\ell_{d+1}})$. 
\end{proof}

\begin{proof}[Proof of lemma \ref{general_rep}] We prove identity (\ref{eq:kindofnewton}) by induction on $d\geq N$.

\smallskip

(A) We start with the case $d=N$. In that case (\ref{eq:kindofnewton}) reduces to
\begin{align} \phi(x)&= \sum_{K\in \binom{\HH_{N-1}}{N-1}} P^{[N-1]}_K(x) \phi(\vartheta_{K\cup \ell_N}) +\sum_{K\in \binom{\HH_{N}}{N-1}} P^{[N]}_K(x) \phi(\nn_K)\\
&= \phi(\vartheta_{\HH_N})+ \sum_{i=1}^N \frac{\ell_i(x)}{\tilde{\ell_i}(\nn_{\HH_N\setminus\ell_i})}  \phi(\nn_{\HH_N\setminus \ell_i}).\end{align}
Since $\phi$ is a linear form, the claim follows from de Boor's identity (\ref{coordinate}).

\par \smallskip

(B) We assume that (\ref{eq:kindofnewton}) holds true for $d$ and prove it for $d+1$. Take $\phi$ a symmetric $(d+2-N)$-linear form. Fix $y\in \RR^N$ and define $\phi_y$ on $(\RR^N)^{d+1-N}$ by $\phi_y(v_1,\dots,v_{d+1-N})=\phi(v_1,\dots,v_{d+1-N},y)$. 
Thus $\phi_y$ is a symmetric $(d+1-N)$-linear form to which we may apply 
the induction hypothesis to get
\begin{equation} \phi_y(x^{d-N+1})=\sum_{i=N}^{d+1}\sum_{K\in \binom{\HH_{i-1}}{N-1}} P^{[i-1]}_K(x) \phi_y \left(x^{d-i},\vartheta_{K\cup \ell_i}, \nn_K^{i-N}\right).  
\end{equation}
Putting $y=x$ in the above expression, we obtain
\begin{equation}\label{eq:whatweknow} \phi(x^{d-N+2})=\sum_{i=N}^{d+1}\sum_{K\in \binom{\HH_{i-1}}{N-1}} P^{[i-1]}_K(x) \phi_x\left(x^{d-i},\vartheta_{K\cup \ell_i}, \nn_K^{i-N}\right).  
\end{equation}
On the other hand, we need to prove 
\begin{equation}\label{eq:whatwewant} \phi(x^{d-N+2})=\sum_{i=N}^{d+2}\sum_{K\in \binom{\HH_{i-1}}{N-1}} P^{[i-1]}_K(x) \phi\left(x^{d+1-i},\vartheta_{K\cup \ell_i}, \nn_K^{i-N}\right).  
\end{equation}

Expressions (\ref{eq:whatweknow}) and (\ref{eq:whatwewant}) differ only for $i=d+1$ and $i=d+2$. Thus, to establish (\ref{eq:whatwewant}), it suffices to prove 
that the term corresponding to $d+1$ in (\ref{eq:whatweknow}) equals the sum of the terms corresponding to $d+1$ and $d+2$ in (\ref{eq:whatwewant}), that is
  \begin{multline}\label{eq:lemmaalg} \sum_{K\in \binom{\HH_{d}}{N-1}} P^{[d]}_K(x) \phi\left(\nn_K^{d+1-N},x\right) \\ =
  \sum_{K\in \binom{\HH_{d}}{N-1}} P^{[d]}_K(x) \phi\left(\vartheta_{K\cup \ell_{d+1}}, \nn_K^{d+1-N}\right) + \sum_{K\in \binom{\HH_{d+1}}{N-1}} P^{[d+1]}_K(x) \phi\left(\nn_K^{d+2-N}\right).
  \end{multline}
  For $K\in \binom{\HH_{d}}{N-1}$ and $x\in \RR^N$,  using de Boor's identity (\ref{coordinate}) with $H=K\cup \ell_{d+1}$, we may write
  \begin{align*}x&=\vartheta_{K\cup \ell_{d+1}} + \sum_{\ell\in K\cup \ell_{d+1}}  \frac{\ell(x)}{\tilde{\ell}(\nn_{(K\cup \ell_{d+1})\setminus \ell})} \nn_{(K\cup \ell_{d+1})\setminus \ell} \\
  &= \vartheta_{K\cup \ell_{d+1}} + \frac{\ell_{d+1}(x)}{\tilde{\ell}_{d+1}(\nn_{K})} \nn_{K}+ \sum_{\ell\in K}  \frac{\ell(x)}{\tilde{\ell}\left(\nn_{(K\setminus \ell) \cup \ell_{d+1}}\right)} \nn_{(K\setminus\ell)\cup \ell_{d+1}} .\end{align*}
  
 Substituting $x$ with the above expression in the left hand side of (\ref{eq:lemmaalg}), we arrive to
  \begin{multline}\label{eq:pralgforredu} \sum_{K\in \binom{\HH_{d}}{N-1}} P^{[d]}_K(x) \cdot \phi\left(\nn_K^{d+1-N},x\right)
  =  \sum_{K\in \binom{\HH_{d}}{N-1}} P^{[d]}_K(x) \cdot \phi\left(\vartheta_{K\cup \ell_{d+1}}, \nn_K^{d+1-N}\right) \\
  + \sum_{K\in \binom{\HH_{d}}{N-1}} P^{[d]}_K(x) \,  \frac{\ell_{d+1}(x)}{\tilde{\ell}_{d+1}(\nn_{K})}\, \cdot  \phi\left(\nn_K^{d+2-N}\right) \\
  + \sum_{K\in \binom{\HH_{d}}{N-1}} \sum_{\ell\in K} P^{[d]}_K(x)\, \frac{\ell(x)}{\tilde{\ell}(\nn_{(K\setminus \ell) \cup \ell_{d+1}})}\cdot  \phi\left(\nn_K^{d+1-N},\nn_{(K\setminus\ell)\cup \ell_{d+1}}\right) .  
  \end{multline}
  Now, for $K\in \binom{\HH_d}{N-1}$, we have $$P^{[d]}_K(x)\; \frac{\ell_{d+1}(x)}{\tilde{\ell}_{d+1}(\nn_{K})}=P^{[d+1]}_K(x).$$ Hence, the second term
   on the right hand side of (\ref{eq:pralgforredu}) is 
 \begin{equation} \sum_{K\in \binom{\HH_{d+1}}{N-1} , \; \ell_{d+1}\not\in K} P^{[d+1]}_K(x) \cdot \phi\left(\nn_K^{d+2-N}\right). \end{equation}
 Thus, since $K\in \binom{\HH_{d+1}}{N-1} , \; \ell_{d+1}\in K$ means $K=K'\cup\{\ell_{d+1}\}$ with $K'\in \binom{\HH_{d}}{N-2}$, to prove (\ref{eq:lemmaalg}), it remains to establish 
\begin{multline}\label{needtoprove}
\sum_{K\in \binom{\HH_{d}}{N-1}} \sum_{\ell\in K} P^{[d]}_K(x) \, \frac{\ell(x)}{\tilde{\ell}(\nn_{(K\setminus \ell) \cup \ell_{d+1}})} \cdot \phi\left(\nn_K^{d+1-N},\nn_{(K\setminus\ell)\cup \ell_{d+1}}\right) \\ = \sum_{K'\in \binom{\HH_{d}}{N-2}} P^{[d+1]}_{K'\cup\ell_{d+1}}(x) \cdot \phi\left(\nn_{K\cup \ell_{d+1}}^{d+2-N}\right).\end{multline}

We first concentrate on the term $P^{[d]}_K(x) \frac{\ell(x)}{\tilde{\ell}(\nn_{(K\setminus \ell) \cup \ell_{d+1}})}$ on the left hand side of (\ref{needtoprove}). 
Since $\ell\in K$ we have
\begin{align}\label{eqn4}
P^{[d]}_K(x) \frac{\ell(x)}{\tilde{\ell}(\nn_{(K\setminus \ell) \cup \ell_{d+1}})} &=\left\{ \prod_{h\in\HH_d\setminus K} \frac{h(x)}{\tilde{h}(\nn_K)}\right\}
\cdot \frac{\ell(x)}{\tell(\nn_{(K\setminus \ell)\cup\ell_{d+1}})}\\
&=\left\{\prod_{h\in \HH_{d}\setminus(K\setminus \ell)}\frac{h(x)}{\tilde h(\nn_{(K\setminus \ell)\cup\ell_{d+1}})}\right\} \cdot \left\{\prod_{h\in \HH_d\setminus K}  \frac{\tilde h(\nn_{(K\setminus \ell)\cup\ell_{d+1}})}{\tilde{h}(\nn_K)}\right\}\\
&= P^{[d+1]}_{(K\setminus \ell)\cup\ell_{d+1}}(x)\cdot \tilde{P}^{[d]}_K(\nn_{(K\setminus \ell)\cup\ell_{d+1}}).
\end{align} 
Using this expression in the left hand side of (\ref{needtoprove}), we arrive at
\begin{multline}\label{eqn5}
\sum_{K\in\binom{\HH_d}{N-1}}\sum_{\ell\in K} P_K^{[d]}(x)\, \frac{\ell(x)}{\tell(\nn_{(K\setminus \ell)\cup\ell_{d+1}})} \cdot \phi\left(\nn_K^{d-N+1},\nn_{(K\setminus \ell)\cup\ell_{d+1}}\right)\\
=\sum_{K\in\binom{\HH_d}{N-1}}\sum_{\ell\in K}  P^{[d+1]}_{(K\setminus \ell)\cup\ell_{d+1}}(x) \; \tilde{P}^{[d]}_K(\nn_{(K\setminus \ell)\cup\ell_{d+1}})  \cdot \phi\left(\nn_K^{d-N+1},\nn_{(K\setminus \ell)\cup\ell_{d+1}}\right) \\
=\sum_{K'\in\binom{\HH_d}{N-2}}P^{[d+1]}_{K'\cup\ell_{d+1}}(x)  \sum_{K'\subset K\in\binom{\HH_d}{N-1}} \tilde{P}^{[d]}_K(\nn_{K'\cup\ell_{d+1}}) \cdot  \phi\left(\nn_K^{d-N+1},\nn_{K'\cup\ell_{d+1}}\right).
\end{multline}
 
 Now, for a fixed $K'\in\binom{\HH_d}{N-2}$, using lemma \ref{th:techobser} for the first equality (we add null terms) and  corollary \ref{homo_rep} for the second one, we get 
\begin{multline}\label{eqn6}
\sum_{K'\subset K\in\binom{\HH_d}{N-1}} \tilde{P}^{[d]}_K(\nn_{K'\cup\ell_{d+1}})  \phi(\nn_K^{d-N+1},\nn_{K'\cup\ell_{d+1}})\\
=\sum_{K\in\binom{\HH_d}{N-1}} \tilde{P}^{[d]}_K(\nn_{K'\cup\ell_{d+1}})  \phi(\nn_K^{d-N+1},\nn_{K'\cup\ell_{d+1}})
= \phi(\nn_{K'\cup\ell_{d+1}}^{d-N+2}).
\end{multline}
Using (\ref{eqn6}) in the last term of (\ref{eqn5}), we finally arrive at 
\begin{multline}\label{eqn7}
 \sum_{K\in\binom{\HH_d}{N-1}}\sum_{\ell\in K} P_K^{[d]}(x)\frac{\ell(x)}{\tell(\nn_{(K\setminus \ell)\cup\ell_{d+1}})} \phi(\nn_K^{d-N+1},\nn_{(K\setminus \ell)\cup\ell_{d+1}})\\
=\sum_{K'\in\binom{\HH_d}{N-2}}P^{[d+1]}_{K'\cup\ell_{d+1}}(x) \cdot \phi\left(\nn_{K'\cup\ell_{d+1}}^{d-N+2}\right),
\end{multline}
which is (\ref{needtoprove}). This completes the proof of the lemma.
\end{proof}
\begin{corollary}\label{error_taylor}
Let $\HH=\{\ell_1,\ldots,\ell_d\}$ be a collection of $d\geq N$ hyperplanes in general position  in $\RR^N$. For every function $f$ of class $C^{d-N+1}$ on a convex neighborhood $\Omega$ of the origin in $\RR^N$ we have
\begin{multline*}
f(x)-\tay^{d-N}_0(f)(x)\\
= \sum_{i=N}^{d+1}\sum_{K\in \binom{\HH_{i-1}}{N-1}} P^{[i-1]}_K(x)\,\cdot\, \int\limits_{\Big[\underbrace{0,\ldots,0}_{d-N+1}\,,\, x\Big]} f^{(d-N+1)}(\cdot)\big(x^{d-i},\vartheta_{K\cup\ell_i}\,,\,\nn_K^{i-N}\big), \quad x\in \Omega.
\end{multline*}
\end{corollary}
\begin{proof}
The remainder formula for Taylor polynomial (as a special case of Kergin interpolation, see e.g. \cite[theorem 3]{micchelli}) gives us,   
\begin{equation} f(x)-\tay^{d-N}_0(f)(x)
=\big[\underbrace{0,\dots,0}_{d-N+1},\, x\, |\, \underbrace{x,\dots,x}_{d-N+1}\big]f=\int\limits_{[0,\dots,0,x]} f^{(d-N+1)}(\cdot)(x,\ldots,x).\end{equation}
The corollary then follows directly from lemma \ref{general_rep} since, for every $a\in \Omega$, $f^{(d-N+1)}(a)$ is a symmetric $(d-N+1)$-linear form on $\RR^N$.
\end{proof}

\subsection{Proof of theorem \ref{th:main}} Let $\Omega$ be a neighborhood of the origin on which $f$ is of class $C^{d+1}$. 
 We may assume that 
\begin{enumerate}
	\item[(i)] $\Omega$ contains $B(0,R)$, the closed euclidean ball of center the origin and radius $R$ and, in view of condition $(C1)$,
	\item[(ii)] all the points of $\Theta^{(s)}=\Theta_{\HH^{(s)}}$ lie in $B(0,R)$, $s\in \NN$.
\end{enumerate}
  We set 
 \begin{equation} M=\max_{a\in B(0,R)} \|f^{(d-N+1)}(a)\| <\infty, \end{equation}
 where $\| \cdot \|$ here denotes the usual norm of a multilinear form. 
We use condition (C2) in the form given by (\ref{eq:C2withdet}) taking (\ref{eq:defnK}) into account as follows.
\begin{enumerate}
	\item[(iii)]  There exists $\delta>0$ such that
\begin{equation}\label{eq:usecondC2}\left| \langle \nn_i, \nn_K\rangle\right|\geq \delta, \quad  K\in \binom{\HH^{(s)}}{N-1},\quad \ell_i\not\in K,\quad s\in \NN. \end{equation} 
\end{enumerate}

\medskip

(A) We first derive an estimate on the polynomials $P_K=P_K^{[d]}$ defined in (\ref{eq:defPKdBoor}). We claim that
\begin{equation}\label{eq:estpk}|P_K^{[d]}(x)|\leq \left (\frac{2R}{\delta} \right )^{d-N+1}, \quad x\in B(0,R), \quad K\in\binom{\HH^{(s)}}{N-1}, \quad s\in \NN. \end{equation}
Indeed, if $K\in \binom{\HH^{(s)}}{N-1}$ and $\ell_i\in \HH^{(s)}\setminus K$, since $\vartheta_{K\cup\ell_i}\in \ell_i$, we have $$|c_i|=|\langle \nn_i, \vartheta_{K\cup\ell_i}\rangle|\leq \|\vartheta_{K\cup\ell_i}\| \leq R.$$ 
Next, using (\ref{eq:usecondC2}) and $\|\nn_i\|=1$, we have
\begin{equation}\left| \frac{\ell_i(x)}{\tell_i(\nn_K)}\right | \leq \frac{|\langle \nn_i,x\rangle|+|c_i|}{|\langle \nn_i,  \nn_K \rangle|}\leq\frac{2R}{\delta},\quad  \ell_i\in \HH^{(s)}\setminus K, \end{equation}
which readily implies (\ref{eq:estpk}).

\medskip

(B) We now use theorem \ref{errordeboor} and corollary \ref{error_taylor} to estimate the difference between a Taylor polynomials and a Chung-Yao interpolation polynomial of a same function. To simplify, we omit the index $s$ in the formulas. We have 
\begin{multline}\label{maineqn1}
\lag[\Theta;f](x)-\tay^{d-N}_0(f)(x)=\left[f(x)-\tay^{d-N}_0(f)(x)\right]-\left[f(x)-\lag[\Theta ;f](x)\right]\\
=\sum_{K\in\binom{\HH_d}{N-1}} P^{[d]}_K(x)\Big([0,\dots,0,x\, |\,\nn_K,\ldots,\nn_K]f-[\Theta_{K},x\,|\,\nn_K,\cdots,\nn_K]f\Big )\\ 
+\sum_{i=N}^{d}\sum_{K\in \binom{\HH_{i-1}}{N-1}} P^{[i-1]}_K(x)\int_{[0,\ldots,0,x]} f^{(d-N+1)}(\cdot)(x^{d-i},\vartheta_{K\cup\ell_i},\nn_K^{i-N}), \quad x\in B(0,R). 
\end{multline}
We call $S_1(x)$ and $S_2(x)$ the terms in the above sum and prove that, for every $x\in B(0,R)$, both of them tends to $0$ as $s\rightarrow\infty$. This will achieve the proof (since  simple convergence on a compact set of nonempty interior implies convergence on any normed vector space topology on $\pol^{d-N}(\RR^N)$).  

\medskip

(C) Since, in view of (\ref{eq:estpk}), the polynomials $P^{[d]}_{K}$ are bounded uniformly in $s$, that $S_1(x)\to 0$ for $x\in B(0,R)$ follows from 
\begin{equation}\Big |[\underbrace{0,\ldots,0}_{d-N+1},x|\underbrace{\nn_K,\ldots,\nn_K}_{d-N+1}]f-[\Theta_{K},x|\nn_K,\cdots,\nn_K]f\Big |\to 0\end{equation}
which is a consequence of the fact that the points of $\Theta=\Theta^{(s)}$ tend to $0$ together with the continuity of the divided differences of $f$ as a function of the two groups of its arguments, see subsection \ref{sec:deBoorrf}.

\medskip

(D) As for the term $S_2(x)$, since the right hand side goes to $0$ as $s\to\infty$, the conclusion follows from the following estimate.  
\begin{equation}\label{claim2}
| S_2(x)|\leq \frac{M}{(d-N+1)!}R^{d-N}\left (1+\frac{2}{\delta} \right )^{d-1}\|\Theta\|,\quad x\in B(0,R),
\end{equation}
where $\|\Theta\|=\|\Theta^{(s)}\|:=\max \{\|\vartheta\|\, :\, \vartheta\in \Theta\}$. 
To prove this,  we observe that if $N\leq i\leq d$ and $K\in \binom{\HH_{i-1}}{N-1}$, the bound (\ref{eq:estpk}) (in which $\HH$ is replaced by $\HH_{i-1}$) gives \begin{equation}|P^{[i-1]}_K(x)|\leq \left (\frac{2R}{\delta} \right )^{i-N}, \quad x\in B(0,R). \end{equation}
   Moreover, for every $a\in B(0,R)$, using $\|\nn_K\|\leq 1$, we have
\begin{equation}
\left | f^{(d-N+1)}(a)(x^{d-i},\vartheta_{K\cup \ell_{i}},\nn_K^{i-N}) \right |
\leq	M\|x\|^{d-i}\|\vartheta_{K\cup \ell_i}\|\|\nn_K\|^{i-N}
\leq MR^{d-i} \cdot \|\Theta\|.
\end{equation}
Hence, since $\mathrm{vol}(\Delta_{d-N+1})=1/(d-N+1)!$, for $x\in B(0,R)$ we have
\begin{equation}
\left| \int_{[0,\ldots,0,x]} f^{(d-N+1)}(\cdot)(x^{d-i},\vartheta_{K\cup\ell_i},\nn_K^{i-N}) \right |
\leq \frac{M}{(d-N+1)!}R^{d-i}\|\Theta\|.
\end{equation}
Combining the above estimates, we obtain
\begin{align}
|S_2(x)| &\leq  \sum_{i=N}^{d}\binom{i-1}{N-1}\left (\frac{2R}{\delta} \right )^{i-N}\frac{M}{(d-N+1)!}R^{d-i} \|\Theta\|\\
&=\frac{M}{(d-N+1)!}\|\Theta\|R^{d-N}\sum_{i=N}^{d}\binom{i-1}{i-N}\left (\frac{2}{\delta} \right )^{i-N}\\
& \leq \frac{M}{(d-N+1)!}\|\Theta\|R^{d-N}\sum_{j=0}^{d-1}\binom{d-1}{j}\left (\frac{2}{\delta} \right )^{j}\\
& = \frac{M}{(d-N+1)!}\|\Theta\|R^{d-N}\left (1+\frac{2}{\delta} \right )^{d-1}.
\end{align}
This concludes the proof of theorem \ref{th:main}.

\subsection{An estimate on the error}

The proof actually yields some estimate on the error between Chung-Yao interpolation polynomials and the Taylor polynomial at the origin. It is shown in the following corollary. 
\begin{corollary}\label{estimatedifference} We assume that the assumptions of theorem \ref{th:main} are satisfied. If $f\in C^{d-N+2}(\Omega)$ then 
\begin{multline} \max_{x\in B(0,R)} \| \lag [\Theta^{(s)}; f](x)-\tay^{d-N}_0(f)(x)\|\\=
O\left(\|\theta^{(s)}\| \cdot \left\{\max_{a\in B(0,R)}\|f^{(d-N+1)}(a)\|+ \max_{a\in B(0,R)}\|f^{(d-N+2)}(a)\|\right\}\right). \end{multline}
where the constant involved in the symbol $O$ does not depend on $f$.   
\end{corollary}
\begin{proof} We turn to the term $S_1(x)$ in the previous proof. For simplicity, we set $m=d-N+1$. Since $f\in C^{m+1}(\Omega)$, for all $K\in \binom{\HH_d}{N-1}$, the mean value inequality gives 
\begin{multline}\label{estimatecorollary}
 \left| [0,\ldots,0,x|\nn_K,\ldots,\nn_K]f-[\Theta_{K},x|\nn_K,\ldots,\nn_K]f  \right|\\
= \left | \int_{\Delta_{m}}
 \left\{ f^{(m)}\Big (x+\sum_{j=1}^{m}(0-x)\xi_j\Big)(\nn_K^m)
-f^{(m)}\Big (x+\sum_{j=1}^{m}(\theta_{Kj}-x\big )\xi_j\Big)(\nn_K^m)  \right \} d\xi \right |\\
\leq \int_{\Delta_{m}}\max_{B(0,R)}\|f^{(m+1)}\| \,  \left\| \sum_{j=1}^{m}\theta_{Kj}\xi_j\right\| \, \|\nn_K\|^{m}d\xi 
\leq \frac{1}{m!} \, \max_{B(0,R)}\|f^{(m+1)}\| \, \|\Theta\|,
\end{multline}
where $\Theta_{K}=\{\theta_{Kj}: i=1,\ldots,m\}$. 
Using (\ref{claim2}) and (\ref{estimatecorollary}) in (\ref{maineqn1}), we finally get 
\begin{multline}\label{finalcorollary} 
 \max_{x\in B(0,R)}\| \lag[\Theta,f](x)-\tay^{d-N}_0(f)(x)\| \\
\leq \binom{d}{N-1}\left (\frac{2R}{\delta} \right )^{d-N+1}\frac{1}{(d-N+1)!} \max_{B(0,R)}\|f^{(d-N+2)}\| \; \|\Theta\|\\+
\frac{1}{(d-N+1)!} \max_{B(0,R)}\|f^{(d-N+1)}\| \; R^{d-N} \; \left (1+\frac{2}{\delta} \right )^{d-1}\|\Theta\|\\
=\Big (M_1\max_{B(0,R)}\|f^{(d-N+1)}\|+  M_2\max_{B(0,R)}\|f^{(d-N+2)}\|\Big ) \; \|\Theta\|.
\end{multline}
\end{proof}
 
\subsection*{Acknowledgement} The work of Phung Van Manh is supported by a PhD fellowship from the Vietnamese government. 
\bibliographystyle{plain}
\bibliography{bib_ch_jpcpvm}

\end{document}